 \documentclass[a4paper,12pt]{amsart}
 \usepackage{amsmath,amsthm,amssymb}
 \usepackage[all]{xy}
\usepackage{color}
\usepackage[margin=2cm]{geometry}

\theoremstyle{definition}
\newtheorem{Thm}{Theorem}[section]
\newtheorem{Def}[Thm]{Definition}

\newtheorem{Cor}[Thm]{Corollary}
\newtheorem{Pro}[Thm]{Proposition}
\newtheorem{Rem}[Thm]{Remark}
\newtheorem{Ex}[Thm]{Example}
\newtheorem{Lem}[Thm]{Lemma}

\newcommand{\Q}{\mathbb{Q}}

\newcommand{\Z}{\mathbb{Z}}

\newcommand{\F}{\mathbb{F}}
\newcommand{\row}{\mathrm{Row}}
\newcommand{\RZ}{\mathbb{R}\mathcal{Z}}
\newcommand{\uxa}{\ensuremath{(\underline{X},\underline{A})}}
\newcommand{\wuxa}{\ensuremath{\widehat{(\underline{X},\underline{A})}}}

\def\err#1{{\textcolor{red}{}}}
\def\red#1{{\textcolor{black}{#1}}}

\newcommand{\e}{\mathbf{e}}
\newcommand{\atopp}[2]{\genfrac{}{}{0pt}{}{#1}{#2}}
\newcommand{\qqed}{\hfill\Box}

\title{Homotopy decomposition of a suspended real toric space}
\date{\today}
\author[S.Choi]{Suyoung Choi}
\thanks{The first named author was supported by Basic Science Research Program through the National Research Foundation of Korea(NRF) funded by the Ministry of Education(NRF-2011-0024975).}
\address{Department of Mathematics, Ajou University, San 5, Woncheondong, Yeongtonggu, Suwon 443-749, Korea}
\email{schoi@ajou.ac.kr}
\author[S.Kaji]{Shizuo KAJI}
\thanks{The second named author was partially supported by KAKENHI, Grant-in-Aid for Young
     Scientists (B) 26800043 and JSPS Postdoctoral Fellowships for Research Abroad.}
\address{Department of Mathematical Sciences,
Faculty of Science,
Yamaguchi University,
1677-1, Yoshida, Yamaguchi 753-8512, Japan}
\email{skaji@yamaguchi-u.ac.jp}
\author[S.Theriault]{Stephen Theriault}
\address{School of Mathematics, University of Southampton,
    Southampton SO17 1BJ, United Kingdom}
\email{S.D.Theriault@soton.ac.uk}

\dedicatory{In memory of Professor Samuel Gitler}

\begin{document}

\begin{abstract}
We give $p$-local homotopy decompositions of the suspensions of real toric spaces for odd primes $p$.
Our decomposition is compatible with the one given by Bahri, Bendersky, Cohen, and Gitler for the suspension of the corresponding real moment-angle complex,  or more generally, the polyhedral product.
As an application, we obtain a stable rigidity property  for real toric spaces.
This is a generalized version of the published paper \cite{SSS}.
\end{abstract}

\maketitle

\section{Introduction}

For a simplicial complex $K$ on $m$-vertices $[m]=\{1, \ldots, m\}$ the
\emph{real moment-angle complex} $\RZ_K$ (or \emph{the polyhedral product}
$(\underline{D^1},\underline{S^0})^K$) of $K$ is defined as follows:
\begin{align*}
    \RZ_K &= (\underline{D^1},\underline{S^0})^K \\
          &:= \bigcup_{\sigma\in K} \left\{(x_1,\dotsc,x_m)\in (D^1)^m \mid x_i \in S^0 \text{ when }i\notin \sigma\right\},
\end{align*}
    where $D^1=[0,1]$ is the unit interval and $S^0 = \{0,1\}$ is its boundary.
    It should be noted that
    $\RZ_K$ is a topological manifold if $K$ is a simplicial sphere \cite[Lemma~6.13]{BP}.

There is a canonical $\F_2^m$-action on $\RZ_K$ which comes from
     the $\F_2$-action on the pair $(D^1,S^0)$.
Any subgroup of $\F_2^m$ is specified as the kernel of a linear map
$\lambda\colon \F_2^m \to \F_2^n$ for some $n\le m$.
The quotient space $\RZ_K/\ker\lambda$ is denoted by $M(K,\lambda)$
 and called the {\em real toric space} associated to $K$ and $\lambda$.
 We denote the $i$-th column of $\lambda$ by $\lambda(i)$.

Real toric spaces are a generalization to well-studied classes of spaces.
When $\lambda$ satisfies the following condition, it is called a (mod $2$) \emph{characteristic function}
of $K$:
\begin{equation}\label{characteristic}
     \text{$\lambda(i_1),\dotsc,\lambda(i_\ell)$ are linearly independent in $\F_2^n$ if $\{i_1, \ldots, i_\ell\} \in K$. }
\end{equation}
%For convenience, a characteristic function $\lambda$ is frequently represented by an $(n \times m)$ $\F_2$-matrix $\lambda = \left(\lambda(1)\;\dotsb\;\lambda(m)\right)$, called a \emph{characteristic matrix}. Define a map $\theta\colon [m] \to \F_2^m$ so that $\theta(i)$ is the $i$-th coordinate vector of $\F_2^m$.
%Then the homomorphism $\lambda$ (viewed as a matrix multiplication) satisfies $\lambda \circ \theta = \lambda$. We will see in Lemma \ref{lem:free-act} that Condition~(\ref{characteristic}) ensures that the group $\ker \lambda \cong \F_2^{m-n}$ acts freely on $\RZ_K$.
\begin{Lem}\label{lem:free-act}
The action of
$\ker \lambda$ on $\RZ_K$ is free if and only if
\eqref{characteristic} is satisfied.
\end{Lem}
\begin{proof}
Let $x=(x_1,x_2,\ldots,x_m)\in \RZ_K = (\underline{D^1},\underline{S^0})^K$
be the fixed point of an element $g=(g_1,g_2,\ldots,g_m)^t\in \ker \lambda \subset \F_2^m$.
This means either $g_i=0$ or $x_i \in (D^1)^{\F_2} = \{1/2\}$ for all $i\in [m]$.
Let $\sigma\in K$ be a simplex such that $x\in (\underline{D^1},\underline{S^0})^\sigma$
and $\lambda_\sigma$ be the sub-matrix of $\lambda$
consisting of columns corresponding to $\sigma$.
Let $g_\sigma$ be the sub-vector of $g$ corresponding to $\sigma$.
Since $g\in \ker \lambda$, we have
\[
\lambda g =\lambda_\sigma g_\sigma + \lambda_{[m]\setminus \sigma} g_{[m]\setminus \sigma} = 0.
\]
Since $\F_2$ acts on $S^0$ freely, we have $g_i=0$ for $i\notin \sigma$.
Then, by the previous equation we have $\lambda_\sigma g_\sigma=0$.
Therefore, the condition \eqref{characteristic} implies $g_\sigma=0$ and $g=0$.
On the other hand, if \eqref{characteristic} does not hold for a simplex $\sigma\in K$,
the point $x_i=\begin{cases} 1/2 & (i\in \sigma) \\ 0 & (i\not\in \sigma) \end{cases}$
is fixed by an element in $\ker\lambda$.
\end{proof}
For a characteristic function $\lambda$, the corresponding real toric space
$M(K,\lambda)$ is known as a \emph{small cover}~\cite{DJ91} if $K$ is a polytopal $(n-1)$-sphere,
and as a \emph{real topological toric manifold}~\cite{IFM12} if $K$ is a star-shaped $(n-1)$-sphere.
\medskip

In \cite[Theorem 2.21]{BBCG} it is shown that there is a homotopy equivalence
\begin{equation}\label{decom_RZ}
\Sigma \RZ_K \simeq \Sigma \bigvee_{I\in [m]} \Sigma\vert K_I\vert
\simeq \Sigma \bigvee_{I\not\in K} \Sigma\vert K_I\vert,
\end{equation}
where $K_{I}$ is the full subcomplex of $K$ on the vertex set $I$ and
$\vert K_{I}\vert$ is its geometric realization so that $|K_I|$ is contractible when $I\in K$.
In this short note, we give
an analogous odd primary decomposition of the suspension of $M(K,\lambda)$.
\begin{Thm}\label{thm:main}
Let $M(K,{\lambda})$ be a real toric space. Localized at an odd prime $p$ or the
rationals (denoted by $p=0$) there is a homotopy equivalence
$$
\Sigma (M(K,\lambda)) \simeq_p \Sigma\!\!\!\!
    \bigvee_{I\in \row(\lambda)}\!\!\!\!\Sigma\vert K_I\vert,
$$
where $\row(\lambda)$ is the space of $m$-dimensional $\F_2$-vectors spanned by the rows of
$\lambda$ associated to~$\lambda$.
We identify vectors in $\F_2^m$ with subsets of $[m]$ in the obvious manner.
\end{Thm}
\begin{Rem}
When $\lambda$ is the identity matrix, $M(K,\lambda)$ is nothing but $\RZ_K$.
On the other hand, when $\lambda$ is zero, $M(K,\lambda)$ is the cone over $|K|$
and contractible.
\end{Rem}

The restriction to odd primes arises because 
when $\vert\ker\lambda\vert$ is inverted in
a coefficient ring $R$ then the quotient map
\(\RZ_K\longrightarrow M(K,{\lambda})\)
induces an injection in cohomology with image the invariant subring
$H^{\ast}(\RZ_{K};R)^{\ker\lambda}$ (see, for example, \cite[\S III.2]{Bredon}).
 This will be used to help
analyze the topology of $M(K,\lambda)$. As $\ker\lambda$
has order a power of $2$ we can take $R$ to be $\mathbb{Z}_{(p)}$ or $\mathbb{Q}$.
In fact, Theorem~\ref{thm:main} fails when $p=2$ in simple cases. For
example, if $K$ is the boundary of a triangle and
$\lambda=\left(\begin{array}{ccc} 1 & 0 & 1 \\ 0 & 1 & 1 \end{array}\right)$,
then $M(K,{\lambda})=\mathbb{R}P^{2}$ but each $\Sigma \vert K_{I}\vert$ is contractible.

Recent work of Yu~\cite{Yu} gave a different decomposition of the suspension
of certain quotient spaces of $\RZ_K$.
He considers $\lambda$
which is associated to a partition on the vertices of $K$,
and proves that $M(K,\lambda)$ decomposes analogously to the Bahri, Bendersky,
Cohen and Gitler decomposition. Yu's decomposition has the advantage of
working integrally, but it has the disadvantage
of working only for particular $\lambda$. Our decomposition, by
contrast, works only after localizing at an odd prime but holds for any $\lambda$.

\subsection*{Acknowledgement}
The authors would like to thank Daisuke Kishimoto for pointing out 
an error in an earlier version of the paper,
and Mikiya Masuda for pointing out that our argument works for not only
characteristic functions but also any linear maps $\lambda$.

\section{Polyhedral product and its stable decomposition}
Let us first recall Bahri, Bendersky, Cohen and Gitler's argument in \cite{BBCG}.
To make it more clear, we present it in its full polyhedral product form.
Let $K$ be a simplicial complex on the vertex set $[m]$ and for $1\leq i\leq m$
let $(X_{i},A_{i})$ be pairs of pointed $CW$-complexes. If $\sigma$ is a face
of $K$ let
\[\uxa^{\sigma}=\prod_{i=1}^{m} Y_{i}\qquad\mbox{where}\qquad
     Y_{i}=\left\{\begin{array}{ll} X_{i} & \mbox{if $i\in K$} \\
              A_{i} & \mbox{if $i\notin K$.}\end{array}\right.\]
The \emph{polyhedral product} is
\[\uxa^{K}=\bigcup_{\sigma\in K}\uxa^{\sigma}.\]
Notice that $\uxa^{K}$ is a subspace of the product $\prod_{i=1}^{m} X_{i}$.
There is a canonical quotient map from the product to the smash product,
\(\prod_{i=1}^{m} X_{i}\longrightarrow\bigwedge_{i=1}^{m} X_{i}\).
The \emph{smash polyhedral product} $\wuxa^{K}$ is the image of
the composite
\(\uxa^{K}\longrightarrow\prod_{i=1}^{m} X_{i}\longrightarrow
   \bigwedge_{i=1}^{m} X_{i}\).
In particular, mapping onto the image gives a map
\(\uxa^{K}\longrightarrow\wuxa^{K}\).

Let $I\subset [m]$. As in~\cite[2.2.3(i)]{DS}, projecting $\prod_{i=1}^{m} X_{i}$ onto $\prod_{i\in I} X_{i}$
induces a map of polyhedral products
\(\uxa^{K}\longrightarrow\uxa^{K_{I}}\).
We then obtain a composition into a smash polyhedral product:
\[p_I: \uxa^{K}\longrightarrow\uxa^{K_{I}}\longrightarrow\widehat{\uxa}^{K_{I}}.\]
Suspending, we can add every such composition over all full subcomplexes of $K$,
giving a composition
\[\overline{H}\colon\Sigma\uxa^{K}\stackrel{comul}{\longrightarrow}\bigvee_{I\subset [m]}\Sigma\uxa^{K_{I}}
     \xrightarrow{\vee \Sigma p_I}\bigvee_{I\subset [m]}\Sigma\widehat{\uxa}^{K_{I}}.\]
Bahri, Bendersky, Cohen and Gitler~\cite[Theorem 2.10]{BBCG} show
that $\overline{H}$ is a homotopy equivalence.

Further, in the special case when
each $X_{i}$ is contractible, they show that there is a homotopy equivalence
$\wuxa^{K_{I}}\simeq\Sigma(\vert K_{I}\vert\wedge\widehat{A}^{I})$
\cite[Theorem 2.19]{BBCG}, where $\widehat{A}^{I}=\bigwedge_{j=1}^{k} A_{i_{j}}$
for $I=(i_{1},\ldots, i_{k})$.
Consequently, when each $X_{i}$ is contractible the map $\overline{H}$
specializes to a homotopy equivalence
\[H\colon\Sigma\uxa^{K}\longrightarrow\bigvee_{I\subset [m]}\Sigma\uxa^{K_{I}}
     \longrightarrow\bigvee_{I\subset [m]}\Sigma\widehat{\uxa}^{K_{I}}
     \stackrel{\simeq}{\longrightarrow}
     \bigvee_{I\subset [m]}\Sigma^{2} (\vert K_{I}\vert\wedge\widehat{A}^{I}).\]

In our case, each pair $(X_{i},A_{i})$ equals $(D^{1},S^{0})$ and $D^{1}$ is
contractible. As there is a homotopy equivalence
$S^{0}\wedge S^{0}\simeq S^{0}$, each $\widehat{A}^{I}$ is homotopy
equivalent to $S^{0}$. Therefore there are homotopy equivalences
\begin{equation}\label{eq:factor}
    \widehat{\RZ}_{K_{I}}:=
    \widehat{(\underline{D^1},\underline{S^0})}^{K_I} \stackrel{\simeq}{\longrightarrow}
\Sigma\vert K_{I}\vert\wedge\widehat{A}^{I}\simeq
    \Sigma\vert K_{I}\vert\wedge S^{0}
    \err{(Removed)\simeq\vert K_{I}\vert\wedge\Sigma S^{0}\simeq \vert K_{I}\vert\wedge S^{1}}
    \simeq\Sigma\vert K_{I}\vert.
    \end{equation}
Thus the map $H$ becomes a homotopy equivalence
\[H\colon\Sigma\RZ_{K}\longrightarrow\bigvee_{I\subset [m]}\Sigma\RZ_{K_{I}}
      \longrightarrow\bigvee_{I\subset [m]}\Sigma\widehat{\RZ}_{K_{I}}
      \stackrel{\simeq}{\longrightarrow}
      \bigvee_{I\subset [m]}\Sigma^{2}\vert K_{I}\vert.\]
It is in this form that we will use the Bahri, Bendersky, Cohen and Gitler
decomposition because, as we will see shortly, it corresponds to a module
decomposition of a differential graded algebra~$R_{K}$ whose cohomology
equals $H^{\ast}(\RZ_K)$. 
%But it is worth pointing out that in \cite[Theorem 2.21]{BBCG}
%it was shown that when each $X_{i}$ is contractible then $\wuxa^{K_{I}}$ is
%contractible if $I\in K$. So the usual Bahri, Bendersky, Cohen and Gitler
%decomposition is of the form
%\[\Sigma\uxa^{K}\simeq
%      \bigvee_{I\notin K}\Sigma^{2} (\vert K_{I}\vert\wedge\widehat{A}^{I}),\]
%giving the special case
%\[\Sigma \RZ_K \simeq \Sigma \bigvee_{I\not\in K} \Sigma\vert K_I\vert,\]
%which is the statement in~\eqref{decom_RZ}.

%%%
\section{Proof of the Main theorem}
Consider the following diagram
\begin{equation}
   \label{qphidgrm}
\xymatrix{
\Sigma \RZ_K \ar[r]^-{\bar{H}} \ar[d]^{\Sigma q} & \Sigma \bigvee_{I\subset [m]} \widehat{\RZ}_{K_{I}}
\ar[r]^{\simeq} &\Sigma\bigvee_{I\subset [m]}\Sigma\vert K_{I}\vert \\
\Sigma M(K,\lambda) & \Sigma \bigvee_{I\in \row(\lambda)} \widehat{\RZ}_{K_{I}}\ar[u]^{\Sigma incl} \ar[l]_-\phi
\ar[r]^{\simeq} &\Sigma\bigvee_{I\in \row(\lambda)}\Sigma\vert K_{I}\vert
}
\end{equation}
where, by definition, $\phi = \Sigma q \circ \bar{H}^{-1} \circ \Sigma incl$.

To prove Theorem~\ref{thm:main} we will show that
$\phi^\ast$ induces an isomorphism on cohomology.
From now on, assume that coefficients in cohomology are $\Q$ or $\Z_{(p)}$,
where $p$ is an odd prime.

First, \red{by~\cite[Theorem 5.1]{Cai}} the cohomology ring of $\RZ_K$ is given as follows.
Let $\Z_{(p)}\langle u_1,\ldots,u_m,t_1,\ldots,t_m \rangle$ be the free associative
algebra over the indeterminants of $\deg u_i=1, \deg t_i=0\ (i=1, \ldots, m)$.
Define a differential graded algebra $R_K$ as the quotient
\[
R_K=\dfrac{\Z_{(p)}\langle u_1,\ldots,u_m,t_1,\ldots,t_m \rangle}
{ (u_i^2,u_iu_j+u_ju_i, u_it_i-u_i,t_iu_i,t_iu_j-u_jt_i,t_i^2-t_i,t_it_j-t_jt_i, u_\sigma \ (\sigma\not\in K))}
\]
where $i\neq j$ and $d(t_i)=u_i$ for each $i=1, \ldots, m$.
Then $H^\ast(\RZ_K)\simeq H^\ast(R_K)$.
We shall use the notation $u_\sigma$ (respectively, $t_\sigma$) for the monomial $u_{i_1} \cdots u_{i_k}$ (respectively, $t_{i_1} \cdots t_{i_k}$) where $\sigma = \{i_1, \ldots, i_k\}$, $i_1 < \cdots < i_k$, is a subset of $[m]$.
For $I\subset [m]$, denote by $R_{K_I}$ the differential graded sub-module of
 $R_K$ spanned by
the monomials
$\{ u_\sigma t_{I\setminus \sigma} \mid \sigma\in K_I\}$. Observe from
the definitions of $R_K$ and $R_{K_{I}}$ that
there is an additive isomorphism $R_K=\bigoplus_{I\subset [m]}R_{K_I}$.

\begin{Lem}
There is an additive isomorphism
\[
H^*(R_{K_I})\simeq \tilde{H}^*(\widehat{\RZ}_{K_{I}})
\]
and the projection
$p_I \colon \RZ_K \to \widehat{\RZ}_{K_{I}}$
induces the inclusion $p_I^\ast: H^*(R_{K_I}) \hookrightarrow H^*(R_K)$.
\end{Lem}
\begin{proof}
The first assertion follows from
$\widehat{\RZ}_{K_{I}}\simeq \Sigma |K_I|$ (see \eqref{eq:factor})
and the isomorphism $H^*(R_{K_I})\simeq \tilde{H}^{*-1}(|K_I|)$ given by
\begin{align*}
R_{K_I} & \to C^*(K_I) \\
u_\sigma t_{I\setminus \sigma}  &\mapsto \sigma^*,
\end{align*}
where $C^*(K_I)$ is the simplicial cochain complex of $K_I$ \red{(\cite[Proposition 3.3]{Cai})}.

To show the second assertion, we look more closely at the
isomorphism $H^\ast(R_K)\simeq H^\ast(\RZ_K)$.
\red{From \cite[\S 3.2]{Cai}},
the monomials $u_\sigma t_{I\setminus \sigma}$ are mapped into
the image of \mbox{$p_I^*: C_e^*(\widehat{\RZ}_{K_I}) \to C_e^*(\RZ_{K})$},
where $C_e^*$ denotes the cellular cochain complex.
By combining this with the first assertion, we deduce the second assertion.
\end{proof}

Now we investigate the maps appearing in (\ref{qphidgrm}).
Since $|\ker \lambda|$ is a unit in the coefficient ring $\Z_{(p)}$,
the map $q^\ast$ is injective with image ${H}^\ast(\RZ_K)^{\ker \lambda}$.
Notice that in cohomology $incl$ induces the projection
$incl^\ast: \bigoplus_{I\subset [m]} H^*(R_{K_I}) \to
\bigoplus_{I\in \row(\lambda)} H^*(R_{K_I})$.
Recall that $\bar{H} = \Sigma \bigvee_{I\subset [m]} p_I\circ comul$
and $\phi=\Sigma q\circ\bar{H}^{-1}\circ\Sigma incl$. So $\phi^{\ast}$
is the composite
\[
\phi^*: H^*(\Sigma M(K,\lambda))\simeq H^*(\Sigma \RZ_K)^{\ker \lambda}
\hookrightarrow H^*(\Sigma \RZ_K)
\xrightarrow{\simeq}\bigoplus_{I\subset [m]} H^*(\Sigma R_{K_I})
\to
\bigoplus_{I\in \row(\lambda)} H^*(\Sigma R_{K_I}),
\]
where $\Sigma$ for graded modules means the degree shift in the positive degree parts.

We aim to show that $\phi^{\ast}$ is an isomorphism. To see this, we need two lemmas.

\begin{Lem}[{\cite[Section~4]{CP}}]\label{lem:CP}
The Reynolds operator
\[
N(x):=\dfrac{1}{|\ker \lambda|}\sum_{g\in \ker \lambda} gx
\]
induces an additive isomorphism
$\bigoplus_{I\in \row(\lambda)} R_{K_I} \xrightarrow{\simeq} R_K^{\ker \lambda}$,
where $R_K^{\ker \lambda}$ is the $\ker \lambda$-invariant ring of $R_K$.
Furthermore, for a monomial $x=u_\sigma t_{I\setminus \sigma}$,
$N(x)$ has the unique maximal term $x$,
where the order is given by the containment of the index set.~$\qqed$
\end{Lem}

\begin{Lem}
\label{Philemma}
The composite
\[
\Phi: R_K^{\ker \lambda}\hookrightarrow R_K \simeq \bigoplus_{I\subset [m]} R_{K_I}
\xrightarrow{\pi} \bigoplus_{I\in \row(\lambda)} R_{K_I}
\]
is an isomorphism of rings, where $\pi$ is the projection.
Note that we consider the ring structure of $\bigoplus_{I\in \row(\lambda)} R_{K_I}$ induced by that of $R_K$.
\end{Lem}
\begin{proof}
We first show this is surjective.
Take an element $x\in \bigoplus_{I\in \row(\lambda)} R_{K_I}$.
We induct on the size of the index set of $x$.
By Lemma \ref{lem:CP}, the terms in $\pi(N(x)-x)\in \bigoplus_{I\in \row(\lambda)} R_{K_I}$
has an index set strictly smaller than that for $x$.
By induction hypothesis, there is an element $y\in R_K^{\ker \lambda}$
such that $\Phi(y)=\pi(N(x)-x)$.
Put $z=N(x)-y\in R_K^{\ker \lambda}$ and
we have $\Phi(z)=\pi(N(x))-\Phi(y)=\pi(x)=x$.

On the other hand, suppose $\Phi(y)=0$ for some $y\in R_K^{\ker \lambda}$.
By Lemma \ref{lem:CP}, there is $x \in \bigoplus_{I\in \row(\lambda)} R_{K_I}$
such that $y=N(x)$ and $y$ must contain the maximal terms in $x$.
Thus, $\Phi(y)=0$ implies $x=0$ and $y=N(x)=0$.
\end{proof}

\begin{proof}[Proof of Theorem~\ref{thm:main}]
Since
$H^*(\RZ_K)^{\ker \lambda}\simeq H^*(R_K^{\ker \lambda})$,
the definitions of $\phi$ and $\Phi$ imply that $\Sigma \Phi^{\ast}=\phi^\ast$.
Therefore, by Lemma~\ref{Philemma}, $\phi^{\ast}$ is an isomorphism.
\end{proof}

\begin{Cor}
We have the ring isomorphism 
\[
H^*(M(K,\lambda))\simeq H^*\left(\bigoplus_{I\in \row(\lambda)} R_{K_I}\right),
\]
where the ring structure of the right hand side is 
induced by that of $R_K$.
\end{Cor}

%\begin{Rem}
%In \cite{CP}, they only consider the case when the coefficients ring is
%either $\Q$ or $\Z/q\Z$
%for $q>2$.
%However, since $H^*(M(K,\lambda)) = H^*(\R Z_K/\ker \lambda) \simeq H^*(\R Z_K)^{\ker \lambda}$
%holds for any coefficients ring such that $2$ is invertible,
%their theorem is valid also for $\Z_{(p)}$ for an odd prime $p$.
%\end{Rem}

\section{Stable rigidity of real toric spaces}
\label{sec:stable}
In this section, we give an application of Theorem~\ref{thm:main}
to a stable rigidity property of real toric spaces.

\begin{Cor}\label{rigidity}
Let $M(K,\lambda)$ be a real toric space over $K$.
When $|K_I|$ for any $I\in \row(\lambda)$
suspends to a wedge of spheres after localization at an odd prime $p$,
$\Sigma M(K,\lambda)$ is homotopy equivalent to a wedge of spheres after localization at $p$.
Let $N(K',\mu)$ be another real toric spaces over $K'$, where
$|K'_I|$ for any $I\in \row(\mu)$ suspends to a wedge of spheres after localization at $p$.
Then, if $H^\ast(M(K,\lambda);\F_p)\simeq H^\ast(N(K',\mu);\F_p)$ as modules, we have
$\Sigma M(K,\lambda) \simeq_p \Sigma N(K',\mu)$.
\err{(Removed)In particular, if $M(K,\lambda)$ is a $\F_p$-homology sphere (or acyclic) over such $K$,
then $M(K,\lambda)$ is homotopy equivalent to a sphere (or a point)
after localization at $p$.}
~$\qqed$
\end{Cor}

\subsection*{Real toric spaces associated to graphs}
Given a connected simple graph $G$ with $n+1$ nodes $[n+1]$,
the {\em graph associahedron} $P_G$ (\cite{CD})
of dimension $n$ is a convex polytope whose facets correspond to the
connected subgraphs of $G$.
Let $K$ be the boundary complex of $P_G$.
We can describe $K$ directly from $G$:
the vertex set of~$K$ consists proper subsets $T\subsetneq [n+1]$
such that $G|_T$ are connected and the simplices are the tubings of $G$.
We define a mod 2 characteristic map $\lambda_G$ on $K$ as follows:
$$
    \lambda_G (T) = \left\{
                      \begin{array}{ll}
                        \sum_{t \in T} \e_t, & \hbox{if $n+1 \not\in T$;} \\
                        \sum_{t \not\in T} \e_t, & \hbox{if $n+1 \in T$,}
                      \end{array}
                    \right.
$$ where $\e_t$ is the $t$-th coordinate vector of $\F_2^n$.
Then we have a real toric manifold $M(G):= M(K,{\lambda_G})$ associated to $G$.
% If $G$ is not connected, then $M(G)$ is defined to be the cartesian product of real toric manifolds of connected components of $G$.

The \emph{signed $a$-number} $sa(G)$ of $G$ is defined recursively by
         $$
            sa(G) = \left\{
                      \begin{array}{ll}
                        1, & \hbox{if $G=\varnothing$;} \\
                        0, & \hbox{if $G$ has a connected component with odd number of nodes} \\
                        -\sum_{T\subsetneq [n+1]}sa(G|_T), & \hbox{otherwise,}
                      \end{array}
                    \right.
        $$
and the \emph{$a$-number} $a(G)$ of $G$ is the absolute value of $sa(G)$.
As shown in \cite{CP12}, there is a bijection $\varphi$ from $\row(\lambda_G)$ to the set of subgraphs of $G$ having an even number nodes
and $|K_I|$ for $I \in \row(\lambda_G)$
is homotopy equivalent to $\bigvee^{a(\varphi(I))} S^{|\varphi(I)|/2 -1}$ where $|\varphi(I)|$ is the number of nodes of $\varphi(I)$.
By Theorem~\ref{thm:main} we obtain the following.
\begin{Cor}\label{cor:graph_case}
We have a homotopy equivalence
\[
\Sigma M(G) \simeq_p \bigvee_{I\in \row(\lambda_G)} \bigvee^{a(\varphi(I))} S^{|\varphi(I)|/2+1} \quad \text{ for any odd prime $p$.}
\]
$\qqed$
\end{Cor}

Now, we define the \emph{$a_i$-number} $a_i(G)$ of $G$ by
$$
    a_i(G) = \sum_{\atopp{T\subseteq [n+1]}{|T|=2i}} a(G|_T).
$$ Then, $a_i(G)$ coincides the $i$-th Betti number $\beta^i (M(G);\F_p)$ of $M(G)$. It should be noted that, by Corollary~\ref{cor:graph_case}, if two graphs $G_1$ and $G_2$ have the same $a_i$-numbers for all $i$'s, then $\Sigma M(G_1) \simeq_p \Sigma M(G_2)$ for any odd prime $p$.

\begin{Ex}
    Let $P_4$ be a path graph of length $3$, and $K_{1,3}$ a tree with one internal node and $3$ leaves (known as a \emph{claw}).
    One can compute $a_i(G):=\sum_{\atopp{T\subseteq [n+1]}{|T|=2i}} a(G|_T)$ as follows:
\begin{align*}
  a_0 (P_4) &= a_0 (K_{1,3}) = 1\\
  a_1 (P_4) &= a_1 (K_{1,3}) = 3\\
  a_2 (P_4) &= a_2 (K_{1,3}) = 2\\
  a_i (P_4) &= a_i (K_{1,3}) = 0 \quad \text{ for $i>2$.}
\end{align*}
    Hence, by Corollary~\ref{cor:graph_case}, $\Sigma M(P_4) \simeq_p \Sigma M(K_{1,3})$ for any odd prime $p$
    although $\Sigma M(P_4)$ and $\Sigma M(K_{1,3})$ are not homotopy equivalent since
they have different mod-$2$ cohomology.
\end{Ex}

%%%
\subsection*{Real toric spaces over fillable complexes}
There is a wide class of simplicial complexes on which every real toric space
satisfies the assumption in Corollary \ref{rigidity}.
\begin{Def}[{\cite[Definition 7.6]{IK}}]
Let $K$ be a simplicial complex.
Let $K_1,\ldots, K_s$ be the connected components of $K$,
and let $\hat{K}_i$ be a simplicial complex obtained from $K_i$ by adding all of its minimal non-faces.
Then $K$ is said to be \emph{$\F_p$-homology fillable} if
(1) for each $i$ there are minimal non-faces $M_1^{i},\ldots, M_r^{i}$ of $K$
 such that $K_i \cup M_1^{i} \cup \cdots \cup M_r^{i}$ is acyclic over $\F_p$, and
(2) $\hat{K}_i$ is simply connected for each $i$.

Moreover, we say that $K$ is {\em totally $\F_p$-homology fillable} when
$K_I$ is $\F_p$-homology fillable for any $\emptyset \neq I \subset [m]$.
\end{Def}

\begin{Pro}[{\cite[Proposition 7.11]{IK}}]
If $K$ is $\F_p$-homology fillable, then $(\Sigma |K|)_{(p)}$ is a wedge of $p$-local spheres.
\end{Pro}

There is a large class of simplicial complexes which are totally homology fillable.
\begin{Pro}[{\cite[Theorem 8.21]{IK}}]
If the Alexander dual of $K$ is sequentially Cohen-Macaulay over $\F_p$ (\cite{BWW}),
then $K$ is totally $\F_p$-homology fillable.
\end{Pro}
Note that
shifted and shellable simplicial complexes are sequentially Cohen-Macaulay
over $\F_p$.

%
%\begin{Def}[\cite{BWW}]
%Let $K^{\langle i \rangle}$ be the sub-complex of $K$ generated by the faces of dimension $\ge i$
%and
%\[
%lk(\sigma)= \{ \tau\in K \mid \tau \cap \sigma = \emptyset, \tau \cup \sigma \in K\}.
%\]
%The simplicial complex $K$ is called sequentially Cohen-Macaulay over
%$\F_p$ if for any $\sigma \in K$ and $i\ge 0$,
%\[
%\tilde{H}_k(lk(\sigma)^{\langle i \rangle};\F_p) = 0 \quad  (\forall k<i).
%\]
%\end{Def}
%
%\begin{Pro}[\cite{IK}]
%If the Alexander dual $K^\vee$ of $K$ is sequentially Cohen-Macaulay (SCM, for short)
%over $\F_p$,
%$\Sigma |K|_{(p)}$ is a wedge of spheres.
%Moreover, since $K_I^{\vee}$ for any $I \subset [m]$ is
%SCM if $K^\vee$ is,
%the assumption in Corollary~\ref{rigidity} is satisfied for the Alexander duals of SCM complexes.
%\end{Pro}

%\section{TODO}
%If we can analyse $q^*$ and $\phi^*$ on the mod-2 cohomology,
%we may upgrade the theorem to
%\[
%\Sigma (M(K,\lambda)) \simeq \Sigma\!\!\!\!
%    \bigvee_{I\in \row(\lambda)}\!\!\!\!\Sigma\vert K_I\vert / \ker \lambda
%\]
%( by the vanishing of Mislin genus for co-H-spaces )
%\end{enumerate}

%\item Find a decomposition of $\Omega M(K,\lambda)$
%( Whitehead products? ) and
%give applications to aspherical $M(K,\lambda)$.
%\end{enumerate}

\end{document}